\documentclass[a4paper,12pt,oneside,reqno]{amsart}
\usepackage{hyperref}
\usepackage[headinclude,DIV13]{typearea}
\areaset{15.1cm}{25.0cm}
\usepackage{txfonts,amssymb,amsmath,amsthm,bbm}

\usepackage[latin1] {inputenc}
\usepackage{graphicx, psfrag}
\usepackage{xcolor}
\usepackage{subfigure}

\newtheorem{theorem}{Theorem}[section]
\newtheorem{lemma}[theorem]{Lemma}

\definecolor{bbm}{RGB}{51,153,0}
\definecolor{above}{RGB}{128,0,128}
\definecolor{below}{RGB}{102,0,204}
\definecolor{cascade}{RGB}{204,0,0}
\definecolor{iid}{RGB}{153,51,0}
\definecolor{darkgreen}{rgb}{0,.6,0}
\definecolor{darkagenta}{rgb}{.5,0,.5}
\definecolor{darkred}{rgb}{1,0,0}
\definecolor{darkblue}{rgb}{0,0,.4}
\definecolor{black}{rgb}{0,0,0}
\definecolor{white}{rgb}{1,1,1}

\theoremstyle{remark}
\newtheorem*{remark}{Remark}

\def\paragraph#1{\noindent \textbf{#1}}

\numberwithin{equation}{section}

\def\d{\mathrm{d}}

\def\<{\langle}
\def\>{\rangle}

\def\a{\alpha}
\def\b{\beta}
\def\e{\varepsilon}

\def\g{\gamma}

\def\l{\lambda}

\def\S{\Sigma}

\def\del{\partial}
\def\R{{\mathbb R}}  
\def\N{{\mathbb N}}  
\def\P{{\mathbb P}}  
\def\Z{{\mathbb Z}}

\def\E{{\mathbb E}}

\def\T{{\mathbb T}}

\let\cal=\mathcal

\def\TT{{\cal T}}

 \def \b {{\beta}}

 \def \g {{\gamma}}
 \def \l {{\lambda}}
 \def \d {{\delta}}
 \def \a {{\alpha}}

 \def \del {{\partial}}

 \def \ba {\begin{array}}
 \def \ea {\end{array}}

 \newcommand{\be}{\begin{equation}}
 \newcommand{\ee}{\end{equation}}

\newcommand{\bea}{\begin{eqnarray}}
 \newcommand{\eea}{\end{eqnarray}}
\def\TH(#1){\label{#1}}\def\thv(#1){\ref{#1}}
\def\Eq(#1){\label{#1}}\def\eqv(#1){(\ref{#1})}

\def\sfrac#1#2{{\textstyle{#1\over #2}}}

 \def \1{\mathbbm{1}}

\def \bz {{\boldsymbol z}}
\def \bx{{\boldsymbol x}}
\def \by{{\boldsymbol y}}

\def\eee{{\mathrm e}}
\def\ddd{{\mathrm d}}

\def\spr{\hspace{-0.05cm}\mathrm{spr}}

\def\rep{\hspace{-0.05cm}\mathrm{rep}}
\def\Dir{\mathrm{Dir}}

\begin{document}
\title[Self-repellent branching random walk]{Self-repellent branching random walk}

\author[A. Bovier]{Anton Bovier}
\address{A. Bovier,
Institute for Applied Mathematics, University of Bonn, 
Endenicher Allee 60, 53115 Bonn, Germany}
\email{bovier@uni-bonn.de}

\author[L. Hartung]{Lisa Hartung}
\address{L. Hartung,
Institut f\"ur Mathematik, Johannes Gutenberg-Universit\"at Mainz,
Stau\-dingerweg 9, 55128 Mainz, Germany}
\email{lhartung@uni-mainz.de}

\author[F. den Hollander]{Frank den Hollander}
\address{F. den Hollander,
Mathematical Institute, Leiden University,
Einsteinweg 55, 2333 CC Leiden, The Netherlands}
\email{denholla@math.leidenuniv.nl}

\date{\today}

\begin{abstract}  
We consider a system of particles performing a discrete-time binary branching random walk with independent standard normal increments subject to a penalty $\b$ for every pair of particles that get within distance $\e$ of each other at every time. We study the optimal configurations that minimise the sum of the spread out cost and the repulsion cost up to a given time horizon $N$. We show that at time $N$ particles are spread out over a distance $\asymp (\b\e)^{1/3} 2^{2N/3}$. We also show that the total cost of the optimal configurations up to time $N$ is $\asymp (\b\e)^{2/3} 2^{4N/3}$.
\end{abstract}

\thanks{AB was supported by the Deutsche Forschungsgemeinschaft (DFG, German Research Foundation) under Germany's Excellence Strategy - GZ 2047/1, Projekt-ID 390685813, and through Project-ID 211504053 - SFB 1060. LH was supported by the Deutsche Forschungsgemeinschaft (DFG, German Research Foundation) through Project-ID 233630050 -TRR 146, Project-ID 443891315 within SPP 2265, and Project-ID 446173099. FdH was supported by the Netherlands Organisation for Scientific Research (NWO) through Gravitation Grant NETWORKS 024.002.003 and by the Alexander von Humboldt Foundation. We thank Stefan M\"uller for interesting discussions on Gamma-convergence that inspired our analysis of the properties of the optimal configurations. AB and FdH thank the Tsinghua Sanya International Mathematics Forum for hospitality.}


\maketitle

\section{Introduction and main results}
\label{intro}


\subsection{Background}
\label{ss.background}

Both branching Brownian motion (BBM) and branching random walk (BRW) can be seen as elementary models for the evolution of a population of individuals that are subject to birth, death, and motion in space \cite{Moyal62,AdkeMoyal63}. Of primary interest in these models is the evolution of the spatial distribution of the population over time. Over the past 40 years, BBM and BRW have been investigated intensively from the point of view of extreme-value theory \cite{B_M,LS,chauvin88,chauvin90,ABK_G,ABK_P,ABK_E,ABBS,CHL17,bbm-book}.

As  models for population dynamics, both BBM and BRW are unrealistic as they lead to uncontrolled exponential growth of the population size and therefore to unsustainably high densities of the population. For instance, the population size of binary BRW at time $n$ grows like $2^n$, while the population spreads out over a volume of order $n$. Several variants have been proposed where, according to some selection rule, offspring is selected in such a way that the total population size remains controlled \cite{BD97,Mallein17,CorMal17,maillard2020}, or competitive interactions between particles are added to suppress the growth \cite{Engl2010,Engl2015,Penington2017,Penington2019}. Other related question that have been considered in the literature concern branching processes \emph{conditioned} on unusual behaviour, e.g.\ staying above a barrier or having exceedingly large maxima. See, e.g.\ \cite{FHL24, BaiHar22} and references therein.

In \cite{BovHar23}, a BBM is considered that is penalised when any two particles get close to each other at any time. It turns out that, given a large time horizon, the optimal strategy for this process is to delay branching until a late random time, to guarantee that the total number of particles generated at the time horizon remains finite. In short, in this model the optimal strategy is to do nothing for a long time, and branch and diffuse only shortly before the time horizon. In the present paper, we look at a similar process, but in a setting where branching cannot be avoided. Specifically, we consider a BRW where all particles duplicate at each unit of time, jump according to independent Gaussian increments, and any two particles that get close to each other at any time receive a penalty. Section~\ref{ss.model} defines the model. Section \ref{ss.notation} introduces notation and states the main theorem. Section~\ref{ss.heuristics} provides the underlying heuristics, while Section~\ref{ss.outline} gives the outline of the remainder of the paper and comments on the proof strategy.  

There is a large literature on self-repellent random processes. Examples  are the Edwards model for weakly self-avoiding Brownian motion and the Domb-Joyce model for weakly self-avoiding random walk.
Further examples are  polymer models;
see e.g.\ \cite{Giacomin2007, denHollander2009} and references therein.

Our choice of Gaussian increments is natural, but can be generalised. Our scaling results do, however, depend on the tail of the distribution of the increments. In the present paper we do not explore this generalisation.


\subsection{Model}
\label{ss.model}

Let 
\be\Eq(BRW.1)
X = \left((X_i(n))_{i=1}^{2^n}\right)_{n\in\N}
\ee 
denote the particles of a discrete-time binary BRW whose increments are i.i.d.\ standard normal random variables. Write $\P$ to denote the law of $X$. For $N \in \N$ and $\e>0$, define 
\be\Eq(int.1)
J_{N,\e}(X) = \sum_{n=1}^N \sum_{ {1 \leq i,j \leq 2^n} \atop {i \neq j} } \1_{|X_i(n)-X_j(n)|<\e},
\ee
which is the total time up to time $N$ that two particles are within distance $\e$ of each other. This quantity can be viewed as the \emph{total collision local time} on length scale $\e$ of all the particles up to time $N$. We are interested in the law of $X$ under the tilted probability measure $P_{N,\b,\e}$ given by the Gibbs measure
\be\Eq(tilt.1)
P_{N,\b,\e}(\cdot) = \frac{1}{Z_{N,\b,\e}} \,\E\left[\1_{X \in\,\cdot\,}\,\eee^{-\b J_{N,\e}(X)}\right],
\ee
with
\be\Eq(part.1) 
Z_{N,\b,\e} = \E\left[\eee^{-\b J_{N,\e}(X)}\right],
\ee
the normalising partition function and $\b>0$ the \emph{inverse temperature}, i.e., all collisions of all particles up to time $N$ are penalised by a factor $\eee^{-\b}$. Note that the birth of particles is \emph{not} penalised: the branching is deterministic and binary. 

It is well known that under the law $\P$ the spread of a \emph{typical} particle configuration at time $n$ is of order $n$. Since there are $2^n$ particles at time $n$, the total collision local time in a typical particle configuration at time $n$ is of order $\e\,2^{2n}/n$. When summed over $n$ this gives a value of $J_{N,\e}$ that is of order $\e\,2^{2N}/N$. 

The situation is \emph{very different} under the law $P_{N,\b,\e}$. Our main theorem below gives a rough description of the most likely configurations of the particles under the law $P_{N,\b,\e}$ for $N$ large and $\b,\e$ fixed. Before stating this theorem we need to introduce some notation. 


\subsection{Notation and main theorem}
\label{ss.notation}

We denote the infinite binary tree by $\T$ and the finite binary subtree with $n$ generations by $\T^{(n)}$. We label the leaves of $\T$ by an infinite multi-index $\bx = (x_1x_2\dots) \in \{0,1\}^\N$, and write $\bx^{(n)}$ for the ancestor of $\bx$ in the $n$-th generation. The root of $\T$ is denoted by $\emptyset$, and $\T$ is endowed with the natural lexicographic order. Let $h\colon\,\T \rightarrow \R$ denote the collection of positions of all the particles in the BRW, i.e., $h(\bx^{(n)})$ is the position of the ancestor of $\bx$ in the $n$-th generation. Then the probability density of a configuration of particles in the first $N$ generations is given by
\be\Eq(prob.1)
p(h;\T^{(N)}) = \exp\left(-S_{\spr}(h,\T^{(N)})\right)
\ee
with \emph{spread out cost} 
\be\Eq(spread.0)
S_{\spr}(h,\T^{(N)}) = \sum_{n=0}^{N-1} W_n(h),
\ee
where 
\be\TH(spread)
W_n(h) = \tfrac12 \sum_{\bz^{(n)}\in \T^{(n)}}
\left(\left(h(\bz^{(n)}0)-h(\bz^{(n)})\right)^2+\left(h(\bz^{(n)}1)-h(\bz^{(n)})\right)^2\right).
\ee
The \emph{repulsion cost} is
\be\Eq(penalty)
S_{\rep}(h,\T^{(N)}) = \b J_{N,\e}(h,\T^{(N)})
\ee
with
\be\Eq(penal.100)
J_{N,\e}(h,\T^{(N)}) = \sum_{n=1}^N I_{n,\e}(h),
\ee
where
\be\TH(penal.101)
I_{n,\e}(h) = \sum_{ {\bz^{(n)},\by^{(n)} \in \T^{(n)}} \atop {\bz^{(n)} \neq \by^{(n)}} } \1_{|h(\bz^{(n)})-h(\by^{(n)})|<\e}.
\ee
Thus, for given $N\in \N$, the optimal strategy will be given by the configurations minimising the functional 
\be\Eq(all.1)
S(h,\T^{(N)}) = S_{\spr}(h,\T^{(N)}) + S_{\rep}(h,\T^{(N)}).
\ee

\begin{theorem}\TH(main.1) 
A configuration $h^*$ that minimises $S(h,\T^{(N)})$ has the following properties:
There are numerical constants $0 < c \leq C < \infty$, independent of $N,\e,\b$, such that 
\be\Eq(main.2)
c \leq \frac{S(h^*,\T^{(N)})}{(\b\e)^{2/3}2^{4N/3}} \leq C.
\ee
\end{theorem}
The second theorem gives information on the minimisers of $S(h,\TT^{(N)})$. 

\begin{theorem}\TH(main.3)
A configuration $h^*$ that minimises $S(h,\T^{(N)})$ has the following properties:
\begin{itemize}
\item [(i)] \hspace{1mm} There is a finite positive constant $c'$   independent of $N,\e,\b$ such that 
particles at time $N$ are spread over an interval whose length $L_N$ satisfies
\be\Eq(main.4)
\frac{L_N}{(\b\e)^{1/3}2^{2N/3}}\geq c'.
\ee
\item [(ii)]\hspace{1mm}  At time $N$, 
\be
\max_{\bx^{(N)}\in \T^{(N)}} \left | h(\bx^{(N)})\right|\leq 2^{2N/3}(\b\e)^{1/3} \sqrt {2NC},
\ee
where $C$ is the constant in \eqv(main.2).
\item [(iii)] \hspace{1mm}
There are no more than $\d 2^N$ particles farther then
\be
c''2^{2N/3}(\b\e)^{1/3} \g/\d
\ee
away from the origin, where $c''$ is a finite constant independent of $N,\d,\e$.
\end{itemize}
\end{theorem}

Note that Theorems~\thv(main.1) and \thv(main.2) say nothing about the \emph{typical} configurations under the Gibbs measure associated with $S(h,\T^{(N)})$, for which we would need to investigate not only the energy but also the entropy of the configurations. 

\begin{remark} The upper bounds on the spread of optimal configurations are weak. We believe that 
the optimal configuration covers an interval of width $C(\b\e)^{1/3}2^{2N/3}$ with particles placed on sites at distances $\e$. Unfortunately, our lower bounds on the cost of spreading particles are not sharp enough to 
show that.
\end{remark}


\subsection{Heuristics and strategy}
\label{ss.heuristics} 

Before closing this introduction, we give a \emph{rough heuristic} for the optimal strategy describing how the particles spread out over space in the course of time. If at time $n$ the particles are evenly distributed over the interval $[-r(n),r(n)]$, then each particle has roughly $\e\,2^n/r(n)$ particles within distance $\e$ (which is possible as long as $\e\,2^n/r(n)>1$), and so $J_{N,\e}$ receives a contribution of order $\e\, 2^{2n}/r(n)$. On the other hand, to spread out the offspring, for all $x \in [-1,1]$, particles that were at a position $x r(n-1)$ at time $n-1$ must move to position $x r(n)$ at time $n$, which entails a spread out cost of order $\exp(\tfrac12 x^2[r(n)-r(n-1)]^2)$ per particle, and hence a total cost of order 
\be
\exp\left(2^n \int_{-1}^1 \ddd x\, \tfrac12 x^2 (r(n)-r(n-1))^2\right)
= \exp\left(2^n \tfrac13 (r(n)-r(n-1))^2\right).
\ee
We would be tempted to think that, for the optimal choice of $r(n)$, both costs must be equal, i.e., 
\be
\b\,\frac{\e\, 2^{2n}}{r(n)} = 2^n \tfrac13 (r(n)-r(n-1))^2.
\ee
Hence 
\be
r(n)^3 = 2^n\, 3\b\e\, \left(1-\frac{r(n-1)}{r(n)}\right)^{-2},
\ee
which would give
\be
r(n)=C\,2^{n/3}, \qquad C^3 = \frac{3\b\e}{(1-2^{-1/3})^2}.
\ee 
However, the above reasoning is \emph{flawed}: it is \emph{not at all} optimal to optimise the cost in each summand separately. Rather, in anticipation of the fact that the process terminates at time $N$, it pays off to \emph{prepare for future expenses at an early stage}, when there are not so many particles yet.

To get a better answer, we must minimise the functional 
\be \Eq(funny.0)
S(r)= \sum_{n=1}^N \left(\b\,\frac{\e\, 2^{2n}}{r(n)} + 2^n \tfrac13(r(n)-r(n-1))^2\right).
\ee
Differentiating w.r.t.\ $r(n)$, we get, for $1<n<N$,
\be
-3\b\e\, 2^{2n}r(n)^{-2} +2^{n+1}(r(n)-r(n-1))-2^{n+2}(r(n+1)-r(n)) = 0
\ee
or
\be\Eq(funny.1)
r(n)-r(n-1) = 3\b\e\, 2^{n-1}r(n)^{-2} + 2(r(n+1)-r(n)),
\ee
while, for $n=N$,
\be
\Eq(funny.2)
r(N)-r(N-1) = 3\b\e\, 2^{N}r(N)^{-2}, 
\ee
and, for $n=1$,
\be
\Eq(funny.3)
r(1) = 3\b\e\, r(1)^{-2} + 2(r(2)-r(1)),
\ee
with $r(0)=0$. Assume that $2^nr(n)^{-2}\ll 1$ for all $1 \leq n \leq N$. Then we can neglect the nonlinear term, to get the recursion
\be\Eq(funny.4)
r(n+1)-r(n) \approx \tfrac12 (r(n)-r(n-1)),
\ee
which is the Laplace equation on a binary tree. The solution is 
\be
\Eq(funny.5)
r(n) \approx 2r(1)(1-2^{-n}),
\ee
with $r(1)$ a free parameter. Inserting this relation into \eqv(funny.0), we get
\be
\Eq(funny.6)
\begin{aligned}
S(r) &\approx \sum_{n=1}^N \left(\b\,\frac{\e\, 2^{2n}}{2r(1)(1-2^{-n})} + 2^n \tfrac43 r(1)^2 2^{-2n}\right)\\
&\approx \tfrac13 \b\e\, 2^{2N+1}\,r(1)^{-1} + \tfrac43 r(1)^2,
\end{aligned}
\ee
which takes its minimum for 
\be\Eq(funny.7)
r(1)^3 = \b\e\, 2^{2N-2}
\ee
Hence, $2^n r(n)^{-2} \asymp 2^n 2^{-4N/3} \leq 2^{-N/3}\ll 1$, so that our initial assumption is indeed verified. Consequently,
\be
S_N(r) \sim \tfrac23 2^{2/3} (\b\e)^{2/3} 2^{4N/3}
\ee
and 
\be
r(N) \sim 2^{1/3} (\b\e)^{1/3} 2^{2N/3}.
\ee
As shown in Theorem~\thv(main.1), up to constants this is the correct answer.

The strategy of the proof of Theorem \thv(main.1) is as follows. An upper bound on the spread out cost is obtained by presenting a candidate optimal configuration. In this configuration, particles up to time $M\approx 2N/3$ move in such a way that at time $M$ there is exactly one particle at the points $\e\{-2^{M-1}-\tfrac12,2^{M-1},\dots,-\tfrac12,\tfrac12,\dots, 2^{M-1}+\tfrac12\}$. Up to time $M$, particles never get closer than $\e$ to each other, and so the repulsion cost is zero. After time $M$, the offspring of all the particles remains at the positions of their parents, and the repulsion cost becomes non-zero. For the lower bound, we consider a toy cost functional that is smaller than the true cost functional and identify the minimiser of the former. 


\subsection{Outline}
\label{ss.outline}

The remainder of this paper is organised as follows. In Section \ref{discropt} we identify the minimisers of the functional $S_{\spr}(h,\T^{(N)})$ in \eqv(spread.0), which \emph{ignores} the repulsive interaction between the particles. This identification involves a Dirichlet problem and leads to sharp estimates on $S_{\spr}(h,\T^{(N)})$. In Section \ref{s.minbound} we derive upper und lower bounds on the functional $S(h,\T^{(N)})$ in \eqv(all.1), which \emph{includes} the repulsive interaction between the particles. The reason why the analysis of $S_{\spr}(h,\T^{(N)})$ is crucial is that, in the strategies that we consider, the effect of the repulsive interaction between the particles is simple, as in the heuristics in Section~\ref{ss.heuristics}. In fact, the bounds in Theorem~\thv(main.1) come from the choice where the profile of the particles at time $N$ is `flat-shaped', respectively, `tent-shaped'.

\section{Optimal strategy without interaction}
\label{discropt}

In this section we identify the optimal strategy for the branching random walk to reach a configuration at time $M \in \N$ where the $2^M$ particles are distributed at positions such that their distance is just large enough to ensure that there is no penalty to yet. In Section~\ref{sec:repr} we represent the optimal strategy as a Dirichlet problem. In Section~\ref{sec:prop} we find the solution to this Dirichlet problem. In Section~\ref{sec:bounds} we use this solution to compute the spread out cost of the minimising configuration. The main steps in the argument are collected in Lemmas~\ref{monotone.1}--\ref{conjecture.1} below. 


\subsection{Representation of the optimal strategy as a Dirichlet problem}
\label{sec:repr}

When we optimise $S_{\spr}(h,\T^{(M)})$ under the \emph{boundary conditions} 
\be\Eq(bdc)
h(\emptyset) = 0, \qquad h(\bz^{(M)}) = u(\bz^{(M)}), \qquad u\colon\,\T^{(M)} \to \R, 
\ee
the solution is the harmonic function with these boundary conditions, i.e., a solution of the \emph{Dirichlet problem} 
\bea\Eq(diri.1)
h(\bz^{(n)} 0) + h(\bz^{(n)} 1) + h(\bz^{(n-1)}) &=& 3h(\bz^{(n)}),\quad 1 \leq n < M,\nonumber\\ 
h(\emptyset) &=& 0,\nonumber\\
h(\bz^{(M)}) &=& u(\bz^{(M)}).
\eea
It will be convenient to express the latter in terms of the increments of $h$. For $\bz\in \T$ and $1 \leq n \leq M$, set
\be\Eq(diri.2)
a(\bz^{(n)}) = h(\bz^{(n)})-h(\bz^{(n-1)}).
\ee
Then
\be\Eq(lisa.101)
h(\emptyset) + \sum_{j=1}^M a(\bz^{(j)})=u(\bz^{(M)}),
\ee
and, for $1\leq n< M$,
\be \Eq(lisa.102)
a(\bz^{(n)})=a(\bz^{(n)}0)+a(\bz^{(n)}1).
\ee
Define, for $0 \leq n \leq M$,
\be\Eq(lisa.104.1)
\S(\bz^{(n)}) = \sum_{\by^{(M-n)}\in \T^{(M-n)}}u(\bz^{(n)}\,\by^{(M-n)}).
\ee

We first state a lemma that ensures that configurations $h$ that minimise the spread out cost \eqv(spread.0) are monotone.

\begin{lemma}\TH(monotone.1)
Let $h$ be a configuration such that there exist $1\leq n\leq N$ and $\bz,\by$ such that 
\be\Eq(monotone.2)
h(\bz^{(n-1)})<h(\by^{(n-1)}) \quad \text{and} \quad h(\bz^{(n)})>h(\by^{(n)}).
\ee
Then there exists a configuration $h^*$ such that 
\be\Eq(monotone.3)
S_{\spr}(h^*,\T^{(N)}) < S_{\spr}(h,\T^{(N)}).
\ee
\end{lemma}

\begin{proof}
We define $h^*$ explicitly. First, for $0 \leq k \leq n-1$, $h^*(\bx^{(k)})= h(\bx^{(k)})$ for all $\bx\in \T$. Also, for all $\bx\in \T$ such that $\bx^{(n)}\neq \bz^{(n)} $ and $\bx^{(n)}\neq \by^{(n)}$, $h^*(\bx^{(k)})= h(\bx^{(k)})$ for all $0 \leq k \leq N$. Finally,  for all $n\leq k\leq N$, we set $h^*(\bz^{(k)})=h(\by^{(k)})$ and $h^*(\by^{(k)})=h(\bz^{(k)})$. The point is that all increments of $h^*$ are the same as those of $h$, except 
\bea
a^*(\bz^{(n)}) &=& h^*(\bz^{(n)})-h^*(\bz^{(n-1)}) =h(\by^{(n)})-h(\bz^{(n-1)}), \nonumber\\
a^*(\by^{(n)}) &=& h^*(\by^{(n)})-h^*(\by^{(n-1)}) =h(\bz^{(n)})-h(\by^{(n-1)}).
\eea
Therefore 
\be
\begin{aligned}
S_{\spr}(h^*,\T^{(N)})-S_{\spr}(h,\T^{(N)})
&=\sfrac 12\left(h(\by^{(n)})-h(\bz^{(n-1)})\right)^2+\sfrac 12\left(h(\bz^{(n)})-h(\by^{(n-1)})\right)^2\\
&\quad -\sfrac 12\left(h(\bz^{(n)})-h(\bz^{(n-1)})\right)^2+\sfrac 12\left(h(\by^{(n)})-h(\by^{(n-1)})\right)^2.
\end{aligned}
\ee
But, for any four numbers $a,b,c,d$ such that $a<b$ and $c<d$,
\be
(c-a)^2 +(b-d)^2 <(a-d)^2+(b-c)^2,
\ee 
and so \eqv(monotone.2) implies \eqv(monotone.3).
\end{proof}

An immediate consequence of Lemma \thv(monotone.1) is that if we look for minimisers of the spread out cost, we only have to consider solutions of the 
Dirichlet problem with boundary conditions $u$ that are monotone in the sense that
\be
u(\bz^{(M)})\geq u(\by^{(M)}) \quad\text{when}\quad \bz^{(M)}\geq  \by^{(M)},
\ee
with respect to the lexicographic distance on the tree. 

\begin{lemma}
\TH(lisa.103)
The unique solution of \eqv(lisa.101)--\eqv(lisa.102) satisfies the recursive equation
\be\Eq(lisa.104)
\S(\bz^{(n)}) = 2^{M-n} \left(\sum_{j=1}^{n-1} a(\bz^{(j)})\right) + \left(2^{M-n}-1\right) a(\bz^{(n)}).
\ee
\end{lemma}

\begin{proof}
Fix $\bz \in \T$ and $0 \leq n \leq M$. Sum \eqv(lisa.101) over the descendants of  $\bz^{(n)}$. This gives
\bea\Eq(lisa.202)
\S(\bz^{(n)})
&=& \sum_{\by^{(M-n)}\in \T^{(M-n)}}
\left(a(\emptyset) + \sum_{j=1}^M a\Big((\bz^{(n)}\,\by^{(M-n)})^{(j)}\Big)\right)\nonumber\\
&=& 2^{M-n} a(\emptyset) + 2^{M-n} \sum_{j=1}^n a(\bz^{(j)}) + \sum_{j=n+1}^M \sum_{\by^{(M-n)} \in \T^{(M-n)}}
a(\bz^{(n)}\,\by^{(j-n)})\nonumber\\
&=& 2^{M-n} a(\emptyset) +2^{M-n}\sum_{j=1}^n a(\bz^{(j)})
+\sum_{j=n+1}^M 2^{M-j} \sum_{\by^{(j-n)}\in \T^{(j-n)}} a(\bz^{(n)}\,\by^{(j-n)})\nonumber\\
&=& 2^{M-n}\,a(\emptyset)
+2^{M-n} \sum_{j=1}^n a(\bz^{(j)}) +\sum_{j=n+1}^M 2^{M-j}\,a(\bz^{(n)})\nonumber\\
&=& 2^{M-n}\,a(\emptyset) + 2^{M-n}\sum_{j=1}^{n-1}a(\bz^{(j)}) + \left(2^{M-n}-1\right)a(\bz^{(n)}),
\eea
where the fourth equality uses \eqv(lisa.102). Since $a(\emptyset)=0$, this gives \eqv(lisa.104).
\end{proof}

Note that 
\be
a(\bz^{(M)}) = h(\bz^{(M)}) -\sum_{j=1}^{M-1} a(\bz^{(j)}).
\ee
The next lemma gives an explicit form of the solution of the recursion in \eqv(lisa.104).
%

\begin{lemma}\TH(solution.1)
Set $b_n = (2^{M-n}-1)^{-1}$. Then 
\be\Eq(solution.2)
h(\bz^{(n)}) = b_n\S(\bz^{(n)}) + \sum_{\ell=1}^{n-1} \frac{b_{n-\ell+1}b_{n-\ell}}{b_{n+1}}\,\S(\bz^{(n-\ell)}),
\ee
and
\bea
\Eq(solution.22)
a(\bz^{(n)}) &=&
 b_n\S(\bz^{(n)}) - \sum_{\ell=1}^{n-1}  \frac{1}{2^\ell-2^{-M+n+1}}b_{n-\ell}
 \S(\bz^{(n-\ell)}).
\eea
\end{lemma}

\begin{proof}
The recursion \eqv(lisa.104) reads
\be
a(\bz^{(n)})=b_n\S(\bz^{(n)})-b_n2^{M-n}\sum_{\ell=1}^{n-1}a(\bz^{(\ell)})
=b_n\S(\bz^{(n)})-b_n 2^{M-n} h(\bz^{(n-1)}),
\ee
or 
\be
h(\bz^{(n)})=b_n\S(\bz^{(n)})-\left(b_n2^{M-n}-1\right)h(\bz^{(n-1)})
=b_n\S(\bz^{(n)})+\frac{b_{n}}{b_{n+1}} h(\bz^{(n-1)}).
\ee
A simple iteration gives that (note that $\S(\emptyset)=0$)
\be\Eq(h.1)
h(\bz^{(n)})
=b_n\S(\bz^{(n)})+\sum_{\ell=1}^{n-1} \frac{b_{n-\ell+1}b_{n-\ell}}{b_{n+1}}
\S(\bz^{(n-\ell)}).
\ee
Hence
\bea
a(\bz^{(n)}) &=& b_n\S(\bz^{(n)})+\sum_{\ell=1}^{n-1} \frac{b_{n-\ell+1}b_{n-\ell}}{b_{n+1}} \S(\bz^{(n-\ell)})\nonumber\\
&& -b_{n-1}\S(\bz^{(n-1)})-\sum_{\ell=1}^{n-2} \frac{b_{n-\ell}b_{n-1-\ell}}{b_{n}} \S(\bz^{(n-1-\ell)})\nonumber\\
&=& b_n\S(\bz^{(n)}) + \sum_{\ell=1}^{n-1}\S(\bz^{(n-\ell)})\left(\frac{b_{n-\ell+1}b_{n-\ell}}{b_{n+1}} -\frac{b_{n-\ell+1}b_{n-\ell}}{b_{n}}\right)\nonumber\\
&=& b_n\S(\bz^{(n)}) - \sum_{\ell=1}^{n-1}  \frac{1}{2^\ell-2^{-M+n+1}}b_{n-\ell}\S(\bz^{(n-\ell)}),
\eea
where we used that 
\be\Eq(future.1)
\frac{b_{n-\ell+1}}{b_{n+1}} -\frac{b_{n-\ell+1}}{b_{n}}
=-\frac{2^{M-n-1}}{2^{M-n-1+\ell}-1}=-\frac{1}{2^\ell-2^{-M+n+1}}.
\ee
This gives \eqv(solution.22).
\end{proof}

As $M\to\infty$, the expressions in \eqv(solution.2)--\eqv(solution.22) simplify to 
\be\Eq(solution.201)
h(\bz^{(n)}) \sim\sum_{\ell=0}^{n-1} 2^{-\ell} b_{n-\ell} \S(\bz^{(n-\ell)}),
\ee
and 
\be\Eq(solution.202)
a(\bz^{(n)}) \sim b_n\S(\bz^{(n)})-\sum_{\ell=1}^{n-1}2^{-\ell} b_{n-\ell}\S(\bz^{(n-\ell)}).
\ee
Note that 
\be
\frac{b_{n-\ell+1}}{b_{n+1}}-2^{-\ell}=2^{n-M}2^{-\ell}\frac {1-2^{-\ell}}{1-2^{-M+n-\ell}}
\ee
with 
\be
0\leq \frac {1-2^{-\ell}}{1-2^{-M+n-\ell}}\leq 1.
\ee
Thus, setting 
\be
\bar h(\bz^{(n)}) = b_n\S(\bz^{(n)})+\sum_{\ell=1}^{n-1} 2^{-\ell} b_{n-\ell} \S(\bz^{(n-\ell)}),
\ee
we get
\be\Eq(error.1)
\bar h(\bz^{(n)}) \leq h(\bz^{(n)}) \leq \bar h(\bz^{(n)})\left(1+2^{-M+n}\right).
\ee

Lemma \thv(solution.1) in principle gives an explicit solution in terms of the boundary conditions. However, it requires the computation of the quantities $\S(\bz^{(n)})$ defined in \eqv(lisa.104.1), which is not straightforward. We next give a full solution for special boundary conditions.


\subsection{Solution of the Dirichlet problem}
\label{sec:prop}

In the following we are interested in the Dirichlet  problem with boundary conditions that are symmetric around zero and correspond to having single particles at distance $\epsilon$. Hence, we look at 
\be
\Eq(bdcond)
u(\bz^{(M)}) = \begin{cases}
\e\sum_{\ell=2}^M 2^{M-\ell} z_\ell + \tfrac12 \e,\quad &z_1=1,\\
-\e\sum_{\ell=2}^M 2^{M-\ell} z_\ell - \tfrac12 \e,\quad &z_1=0.
\end{cases}
\ee
For these boundary conditions we compute $a(\bz^{(n)})$ and $h(\bz^{(n)})$. We write $h^\Dir$ for the solution of the Dirichlet problem with boundary condition \eqv(bdcond). 

\begin{lemma}
\TH(Direxpl)
The solution of the Dirichlet problem with boundary condition \eqv(bdcond) is given by 
\be\Eq(solution.100)
a(\bz^{(n)}) =  \e\,2^{M-n-2}\left(2-n+2\sum_{k=2}^n z_k\right)+O(1),
\ee
and, for $M\gg n$,
\be\Eq(solution.100extra)
h(\bz^{(n)})={\e}\,2^{M-1}\left(n2^{-n-1}+2\sum_{k=2}^n(2^{-k}-2^{-n-1})z_k \right)(1+O(2^{-M+n}).
\ee
In particular,
\be
h(1^{(n)}) \sim
\begin{cases}
{\e}\,2^{M-1}\left(1-2^{-n-1}n\right), &2 \leq n \leq M,\\
{\e}\,2^{M-3}, &n=1.
\end{cases}
\ee
\end{lemma}
\begin{remark}  In the course of the proof we obtain precise formulas for $a$ and $h$, but we state only the 
approximate forms in the lemma.
\end{remark}
\begin{proof} The key step is to compute $\S(\bz^{(n)})$. 
For $z_1=1$,
\bea
\S(\bz^{(n)}) &=&\e \sum_{\by^{(M-n)}\in \T^{(M-n)}}\left(\sum_{\ell=2}^M 2^{M-\ell} z_\ell+2^{-1}\right)\nonumber\\
&=& \e \sum_{\by^{(M-n)}\in \T^{(M-n)}}\left(\sum_{\ell=2}^n 2^{M-\ell} z_\ell+\sum_{\ell=n+1}^M 2^{M-\ell} y_\ell+2^{-1}\right)\nonumber\\
&=& \e\sum_{\ell=2}^n2^{2M-n-\ell}z_\ell + \e\,2^{2M-n-1}\sum_{\ell=n+1}^M2^{-\ell}+\e\,2^{M-n-1}\nonumber\\
&=& \e\,2^{2M-n}\left(\sum_{\ell=2}^n2^{-\ell}z_\ell + 2^{-n-1}\right).
\eea
In particular, 
\be
\S(1) = -\S(0) = \e\,2^{2M-3}.
\ee
Consequently
\be\Eq(h.2)
b_{n-\ell}\S(\bz^{(n-\ell)})=\frac \e 22^M\frac{1}{1-2^{-M+n-\ell}}\left(\sum_{k=2}^{n-\ell}
2^{-k}z_k+2^{-n-1}\right)
\ee
and
\bea
a(\bz^{(n)}) &=& \frac \e 2\frac{1}{1-2^{-M+n}}\left(\sum_{k=2}^n
2^{M-k}z_k+2^{M-n-1}\right)\nonumber\\
&&-\frac \e 2\sum_{\ell=1}^{n-1}\frac{1}{1-2^{-M+n-\ell}}\left(\sum_{k=2}^{n-\ell}2^{M-k-\ell} z_k+2^{M-n-1}
\right)\frac{1}{1-2^{-M+n+1-\ell}}.
\eea
A re-ordering of the summations gives
\bea
a(\bz^{(n)}) &=& \frac \e 22^{M-n-1}\left(\frac{1}{1-2^{-M+n}}-\sum_{\ell=1}^{n-1} c_{M-n+\ell}\right)\nonumber\\
&&+\frac\e2 2^M\sum_{k=2}^n2^{-k}z_k\left(\frac{1}{1-2^{-M+n-1}}-\sum_{\ell=1}^{n-k}2^{-\ell}c_{M-n+\ell}\right)
\eea
with 
\be\Eq(cl.1)
c_m=\frac{1}{(1-2^{-m})(1-2^{-m-1})}.
\ee 
Using Wolfram Alpha, we get
\be\Eq(alpha.1)
\frac{1}{1-2^{-M+n-1}}-\sum_{\ell=1}^{n-k}2^{-\ell}c_{M-n+\ell}=2^{k-n} \frac{1}{1-2^{-M+k-1}}.
\ee
Hence
\bea
a(\bz^{(n)}) &=& \frac \e 22^{M-n-1}\left(\frac{1}{1-2^{-M+n-1}}-\sum_{\ell=1}^{n-1} c_{M-n+\ell}\right)\nonumber\\
&&+\frac\e2 2^{M-n}\sum_{k=2}^nz_k\left(\frac{1}{1-2^{-M+k-1}}\right).
\eea
This can be also written as 
\bea
a(\bz^{(n)}) &=&\frac \e 2 2^{M-n}\sum_{k=1}^n\left(z_k\frac 1{1-2^{-M+k-1}}-\frac 12c_{M-n+k}\right)\nonumber\\
&&+\frac\e2 2^{M-n}\left( c_{M}/2-\frac 1{1-2^{-M}}+\frac{1}{2-2^{-M+n}}\right).
\eea
We readily verify that, uniformly in $0\leq n\leq M$,
\be
a(\bz^{(n)})=\bar a(\bz^{(n)})+O(\e),
\ee
where we set 
\be
\bar a(\bz^{(n)})= \e 2^{M-n-2} \left(-n+2\sum_{k=1}^n z_k\right)
= \e 2^{M-n-1} \left(\sum_{k=1}^n (z_k-\tfrac12)\right).
\ee

Next, we turn to $h$. Combining \eqv(h.1) with \eqv(h.2), we get 
\bea
h(\bz^{(n)})&=&\frac \e 22^M\frac{1}{1-2^{-M+n-1}}\left(\sum_{k=2}^n
2^{-k}z_k+2^{-n-1}\right)\nonumber\\
&&+\frac \e2 2^M \sum_{\ell=1}^{n-1} \frac{b_{n-\ell+1}}{b_{n+1}} \frac{1}{1-2^{-M+n-\ell-1}}\left(\sum_{k=2}^{n-\ell}
2^{-k}z_k+2^{-n+\ell-1}\right)\nonumber\\
&=&\frac \e 22^M\frac{1}{1-2^{-M+n-1}}\left(\sum_{k=2}^n
2^{-k}z_k+2^{-n-1}\right)\nonumber\\
&&+\frac \e2 2^M (1-2^{-M+n})\sum_{\ell=1}^{n-1} 2^{-\ell} 
c_{M-n+\ell}\left(\sum_{k=2}^{n-\ell}
2^{-k}z_k+2^{-n-1+\ell}\right),
\eea
where we use that (recall the definition of $c_m$ in \eqv(cl.1))
\be 
\frac{b_{n-\ell+1}}{b_{n+1}} \frac{1}{1-2^{-M+n-\ell-1}} = 2^{-\ell}(1-2^{-M+n})c_{M-n+\ell}.
\ee
Re-organising terms, we arrive at
\bea
&&h(\bz^{(n)})=\frac \e 22^{M-n-1}\left(\frac{1}{1-2^{-M+n-1}} +(1-2^{-M+n})\sum_{\ell=1}^{n-1}c_{M-n+\ell}\right)
\nonumber\\
&&\quad+\frac \e2 2^M   \left(\frac{1}{1-2^{-M+n-1}}   \sum_{k=2}^n
2^{-k}z_k + (1-2^{-M+n})\sum_{k=2}^{n-1}2^{-k}z_k \sum_{\ell=1}^{n-k} 2^{-\ell}c_{M-n+\ell}\right).
\eea
By \eqv(alpha.1), this is equal to
\bea\Eq(exact.100)
&&\frac \e 22^{M-n-1}\left(\frac{1}{1-2^{-M+n-1}} +(1-2^{-M+n})\sum_{\ell=1}^{n-1}c_{M-n+\ell}\right)
\nonumber\\
&&+\frac \e2 2^M   \Biggl(\frac{1}{1-2^{-M+n-1}}   \sum_{k=2}^{n}
2^{-k}z_k +\frac{ 1-2^{-M+n}}{1-2^{-M+n-1}}\sum_{k=2}^{n-1}2^{-k}z_k 
\nonumber\\
&&-2^{-n}\sum_{k=2}^{n-1} z_k \frac{1-2^{-M+n}}{1-2^{-M+k-1}}\Biggr).
\eea
Even though this looks rather complicated, it is true that $h(\bz^{(M)})=u(\bz^{(M)})$, as it should be. For $1 \leq n\ll M$, 
this simplifies to 
\be
h(\bz^{(n)})= \frac \e2 2^M\left( n 2^{-n-1} +2\sum_{k=2}^{n}2^{-k} z_k 
-2^{-n} \sum_{k=2}^{n} z_k\right)(1+O(2^{-M+n})),
\ee
and hence, for $n\geq 2$, 
\bea
h(1^{(n)})&=& \frac \e2 2^M\left( n 2^{-n-1} +1-n2^{-n}\right)(1+O(2^{-M+n}))\nonumber\\
&=& \frac \e2 2^M\left( 1-n 2^{-n-1} \right)(1+O(2^{-M+n})),
\eea
and
\be
h(1^{(1)}) =\e 2^{M-3}.
\ee
\end{proof}

Next, we want to ensure that the spacing between particles decreases as $n$ increases, for which we need the following lemma.

\begin{lemma}
\TH(space.1)
Suppose that $z_1=1$ and $u(\bz^{(M)}) = \e\sum_{\ell=2}^M 2^{M-\ell} z_\ell + \tfrac12 \e$. If $h$ is a solution of \eqv(diri.1), then, for all $0 \leq n \leq M$, 
\be 
\min_{ {\bz, \by \in \T^{(n)}} \atop {\bz\neq\by} } |h(\bz^{(n)})-h(\by^{(n)})|\geq \e.
\ee
\end{lemma}

\begin{proof}
Let $\bz^{(n)}>\by^{(n)}$. From \eqv(exact.100) we get, for $1 \leq n<M$,
\be
\begin{aligned}
h(\bz^{(n)}) - h(\by^{(n)})
&= \frac \e2 2^M  \Biggl(\frac{1}{1-2^{-M+n-1}} \sum_{k=2}^{n}
2^{-k}(z_k-y_k) + \frac{1-2^{-M+n}}{1-2^{-M+n-1}}\sum_{k=2}^{n-1}2^{-k}(z_k-y_k)\\
&\quad-2^{-n}\sum_{k=2}^{n-1} (z_k-y_k) \frac{1-2^{-M+n}}{1-2^{-M+k-1}}\Biggr)
\geq \frac\e2 2^M \frac{1}{1-2^{-M+n-1}} 2^{-n}\geq \frac{4\e}{3}.
\end{aligned}
\ee
For $n=M$, the result holds since $h(\bz^{(M)})=u(\bz^{(M)})$.
\end{proof}


\subsection{Bounds on the spread out cost}
\label{sec:bounds}

We close this section by deriving bounds on the asymptotics of $S_{\spr}(h^{\Dir},\T^{(M)})$ for $h^{\Dir}$ satisfying \emph{linear} boundary conditions.
 
\begin{lemma}
\TH(cost.101)
Let $h^{\Dir}$ be the solution to the Dirichlet problem \eqv(diri.1) with boundary condition $u$ as in Lemma \thv(space.1).  Then 
\be
\Eq(up-low.101)
\e^2 \frac {2^{2M-4}}{(1-2^{-M+1})^2}\leq S_{\spr}(h^{\Dir},\T^{(M)}) \leq \e^2\frac{(2^{M-1}+\tfrac12)^2}{1-2^{-M}}.
\ee
\end{lemma}
 
\begin{proof}
Obviously, 
\bea
S_{\spr}(h,\T^{(M)}) \geq  \frac 12\sum_{x_1=0}^1 a(x_1)^2 .
\eea
Hence
\be\Eq(low.1)
S_{\spr}(h,\T^{(M)}) \geq \frac{\S(0)^2+\S(1)^2}{2\left(2^{M-1}-1\right)^2}.
\ee
Since $\S(0)=-\S(1)=\e\,2^{2M-3}$, this gives the claimed lower bound. 

Choose any function that satisfies the boundary condition \eqv(lisa.101). For example, if $h(\bz^{(M)})$ is linear with $h(0 \dots 0) = -2^{M-1} - \frac 12$ and $h(1 \dots 1) = 2^{M-1} + \frac 12$, then we can choose
\be
a(0) = -\frac{2^{M-1}+\tfrac12}{1-2^{-M}}2^{-1}, \,\, a(00) = -\frac{2^{M-1}+\tfrac12}{1-2^{-M}}2^{-2}, \,\, \dots, \,\, 
a(0 \dots 0) =-\frac{2^{M-1}+\tfrac12}{1-2^{-M}}2^{-M}.
\ee
Note that 
\be
a(0) + \dots + a(0 \dots 0) = -2^{M-1} - \tfrac12,
\ee
and choose 
\be 
a(x_1\dots x_n) = a(0 \dots 0)\,(-1)^{\sum_{i=1}^n x_i}.
\ee
This choice satisfies the boundary conditions and, since 
\be
\sum_{j=1}^M 2^j \left(\frac{2^{M-1}+\tfrac12}{1-2^{-M}}\right)^2 2^{-2j} = \frac{(2^{M-1}+\tfrac12)^2}{1-2^{-M}},
\ee
gives the claimed upper bound. 
\end{proof}

Lemma \ref{cost.101} shows that, to leading order in $M$, the lower bound is $\e^2\,2^{2M-4}$ and the upper bound is $\e^2\,2^{2M-2}$, which differ by a factor $4$. The next lemma closes this gap.

\begin{lemma}
\TH(conjecture.1)
As $M \to \infty$,
\be
\Eq(conjecture.2)
S_{\spr}(h^{\Dir},\T^{(M)}) = \e^2 2^{2M-3} +\e^2 O(M^2 2^M).
\ee
\end{lemma}

\begin{proof} 
We have that 

\bea
\Eq(anton.2)
 S_{\spr}(h^{\Dir},\T^{(M)})
&=& \sum_{n=1}^M\sum_{\bz^{(n)} \in \T^{(n)}} \left( a(\bz^{(n)})\right)^2= \sum_{n=1}^M\sum_{\bz^{(n)} \in \T^{(n)}} \left(\bar a(\bz^{(n)}+O(\e)\right)^2\nonumber\\
&=& \sum_{n=1}^M\sum_{\bz^{(n)} \in \T^{(n)}} \left(\bar a(\bz^{(n)}\right)^2+ O(\e)
\sum_{n=1}^M\sum_{\bz^{(n)} \in \T^{(n)}} \bar a(\bz^{(n)}).
\eea
Since $|\bar a (\bz^{(n)})|\leq \e 2^{M-n}n$,
the last term is smaller than
\be\Eq(second.sum)
O(\e^2)\sum_{n=1}^M 2^Mn=O(\e^2) 2^MM^2.
\ee
The first summand is equal to
\Eq(anton.2*)
\bea
&&\sum_{n=1}^M\sum_{\bz^{(n)} \in \T^{(n)}}  \left(\e 2^{M-n-1}\sum_{k=1}^n(z_k-\tfrac12)\right)^2\nonumber\\
&&= \frac {\e^2}42^{2M}\sum_{n=1}^M 2^{-2n} \sum_{\bz^{(n)} \in \T^{(n)}} 
\left(\sum_{k,m=1}^n(z_k-\tfrac12(z_m-\tfrac12)\right)\nonumber\\
&&= \frac {\e^2}42^{2M}\sum_{n=1}^M 2^{-2n} \sum_{\bz^{(n)} \in \T^{(n)}} 
\sum_{k=1}^n(z_k-\tfrac12)^2\nonumber
\\
&&= \frac {\e^2}42^{2M}\sum_{n=1}^M 2^{-2n} 2^n \frac n{4}=
\frac {\e^2}{16}2^{2M}\sum_{n=1}^M 2^{-n} n\nonumber\\
&&= \frac {\e^2}{16}\left(2^{2M+1}-(M+2)2^M\right).
\eea
This, together with \eqv(second.sum), gives the assertion.
\end{proof}

\section{Bounds on the minimal cost}
\label{s.minbound}

In this section we prove Theorem~\thv(main.1). In Section~\ref{ss.opint} we find the optimal configurations for the repulsion cost. In Section~\ref{ss.lbspr} we derive a lower bound on the total cost when the size of the spreading is prescribed. In Section~\ref{ss.optfull} we derive an upper bound on the total cost. In Section~\ref{ss.bdspreading} we derive upper and lower bounds on the spreading. The main steps in the argument are collected in Lemmas~\ref{grid.1}--\ref{almost-opt.1} below.


\subsection{Optimal configurations for the repulsion cost.}
\label{ss.opint}

The repulsion cost has the curious feature that minimisers have all their particles nicely sitting on a lattice of spacing $\e$.

\begin{lemma}
\TH(grid.1)
Consider a configuration $h$ for which at time $N$ there are $n_\ell$ particles on the intervals $\e (\ell,\ell+1)$, $\ell\in\Z$. Then the configuration $h^*$ for which there are $n_\ell$ particles at $\e (\ell+\tfrac12)$ and no particles elsewhere 
satisfies 
\be
 I_{N,\e}(h)\geq I_{N,\e}(h^*)=\sum_{\ell\in\Z} n_\ell(n_\ell-1).
\ee
The inequality is strict if in $h$ there are particles in consecutive intervals that are at a distance less than $\e$ from each other.
\end{lemma}

\begin{proof}
Let $h$ be any configuration of the $2^N$ particles. From $h$ we construct a new configuration $h^*$ by moving all points of $h$ in an $\e/2$ neighbourhood of the lattice points $\e(\ell-\tfrac 12)$ to that lattice point, i.e., 
\be\Eq(grid.2)
h^*(\bz^{(N)}) = \e(\ell -\tfrac12) \quad \text{if} \quad |h(\bz^{(N)})- \e(\ell -\tfrac12)| < \tfrac12 \e.
\ee
Then 
\bea\Eq(grid.3)\nonumber
I_{N,\e}(h^*) 
&=& \sum_{ {\bz^{(N)},\by^{(N)}\in \T^{(n)}} \atop {\bz^{(N)}\neq\by^{(N)}} } \1_{|h^*(\bz^{(N)})-h^*(\by^{(N)})|<\e}
= \sum_{ {\bz^{(N)},\by^{(N)}\in \T^{(N)}} \atop {\bz^{(N)}\neq\by^{(N)}} } \1_{h^*(\bz^{(N)})=h^*(\by^{(N)})}\\
&= &\sum_{ {\bz^{(N)},\by^{(N)}\in \T^{(N)}} \atop {\bz^{(N)}\neq\by^{(N)}} } \1_{h^*(\bz^{(N)})=h^*(\by^{(N)})}
\1_{|h(\bz^{(N)})-h(\by^{(N)})|<\e}\nonumber\\
&\leq& \sum_{ {\bz^{(N)},\by^{(N)}\in \T^{(N)}} \atop {\bz^{(N)}\neq\by^{(N)}} } 
\1_{|h(\bz^{(N)})-h(\by^{(N)})|<\e}=I_{N,\e}(h).
\eea
The second equality follows since all points in $h^*$ either coincide or are a distance at least $\e$ apart. The third equality follows since points that coincide in $h^*$ are less than $\e$ apart in the original configuration. The last inequality is strict if there are points in $h$ that are not in the same $\e/2$ neighbourhood of a lattice point but still less than $\e$ from each other. Thus, for any configuration $h$ with $2^N$ points, there is a configuration $h^*$ such that all particles are on the lattice points $\e(\ell -\tfrac12)$, $\ell=-2^{K-1},\dots, 2^{K-1}$, with $n_\ell$ points at the $\ell$-th point such that with $\sum_{\ell=-2^{K-1}}^{2^{K-1}} n_\ell=2^n$. For such a configuration, the interaction energy is equal to
\be\Eq(grid.5)
I_{N,\e}(a) = \sum_{\ell=-2^{K-1}}^{2^{K-1}}  n_\ell(n_\ell-1).
\ee
\end{proof}

As a corollary we obtain the following.

\begin{lemma}
\TH(flat.1)
The repulsion cost under the constraint that all particles at time $n$ are contained in an interval of length $2^K$, $K\leq n$, is minimised if there are $2^{n-K}$ particles at the sites $\e\{-\tfrac12-2^{K-1},\dots \tfrac12+2^{K-1}\}$. This minimal cost is 
\be
2^K2^{N-K}(2^{N-K}-1)=2^{2N-K}(1-2^{K-N}).
\ee
\end{lemma}


\subsection{Lower bound on the total cost}
\label{ss.lbspr}

\begin{lemma} 
\TH(lowerbound.1)
Fix $N\in \N$ sufficiently large. Then the minimal cost of a configuration $h$ 
 satisfies
\be
\Eq(lowerbound.20)
S(h,\T^{(N)})\geq 2^{4N/3} \tfrac23 2^{2/3}(\b\e)^{2/3}(1+o(1)),
\ee
where $o(1)$ tends to zero as $N\to\infty$.
\end{lemma}

\begin{proof} 
The strategy to prove this lower bound is as follows. First, we assume that at time $N$ the $2^N$ particles are distributed such that there are $n_\ell$ particles in the interval $\e[\ell,\ell+1)$, $\ell \in \N_0$, and by symmetry, the same numbers in the intervals $\e[-\ell-1,-\ell)$, $\ell \in \N_0$. Next, we bound the spread out cost from below by using Lemma \thv(lisa.103). Next, we bound the repulsion cost from below by just the contribution from time $N$, which is a function of the boundary condition only. 
Finally, we optimise the sum of these two lower bounds over the possible boundary conditions. This will provide the desired lower bound.

Recall that, by Lemma \thv(lisa.103), we have that
\be\Eq(spread.1)
\min_{h} S_{\spr}(h,\T^{(N)})
\geq a(\bz^{(1)})^2 =\frac {\S(\bz^{(1)})^2}{2^{N-1}-1}
\ee
with 
\be\Eq(lisa.104.01)
\S(\bz^{(1)}) = \sum_{\by^{(N-1)}\in \T^{(N-1)}}u(\bz^{(1)}\,\by^{(N-1)}).
\ee
For a configuration as described above, $\S(1)$ satisfies the bounds
\be
\e\sum_{k=0}^{K^*-1} k n_k
\leq 
\S(1)\leq \e\sum_{k=0}^{K^*-1} (k+1) n_k,
\ee
where $K^*$ is the right bundary of the last occupied intervall. Note that the difference between the upper and the lower bound is only $2^{N-1}$, which will be insignificant. Hence, 
\be
S_{\spr}\geq \e^2  \frac{\left(\sum_{k=0}^{K^*-1}  k n_k\right)^2}{(2^{N-1}-1)^2}.
\ee
Moreover, the repulsion cost is higher than the cost at time $N$, i.e.,
\be
S_{\rep} \geq 2\b \sum_{k=0}^{K^*-1} n_k(n_k-1).
\ee
Hence, we get a lower bound on the total energy by minimising 
\be
\e^2  \frac{\left(\sum_{k=0}^{K^*-1}  k n_k\right)^2}{(2^{N-1}-1)^2}+ 2\b \sum_{k=0}^{K^*-1} n_k(n_k-1)
\ee
over $n_k$ and $K^*$ under the constraint 
\be\Eq(constraint.1)
\sum_{k=0}^{K^*-1} n_k=2^{N-1}
\ee
wit $n_k\geq 0$ for all $k\leq K^*$.

Hence, we get a lower bound on the total energy by minimising 
\be
\e^2  \frac{\left(\sum_{k=0}^{K^*-1}  k n_k\right)^2}{(2^{N-1}-1)^2}+ 2\b \sum_{k=0}^{K^*-1} n_k(n_k-1),
\ee
over $n_k$ and $K^*$ under the constraint \eqv(constraint.1).
\
with $n_k\geq 0$ for all $k< K^*$. We first fix $K^*$ and minimise over the $n_k$.
 This leads to the set of equations 
\be
\e^2 \frac{ \left(\sum_{k=0}^{K^*-1}  k n_k\right)(2\ell +1)}
{(2^{N-1}-1)^2}+ 2\b(2n_\ell-1) -\l=0
\ee
for $\ell=0,\dots,K^*-1$, where we use \eqv(constraint.1). We write this in the form 
\be
C (2\ell +1)+4\b n_\ell-2\b-\l=0,
\ee
with 
\be
C\equiv \e^2 \frac{\sum_{k=0}^{K^*-1} k n_k}{(2^{N-1}-1)^2}.
\ee
Thus,
\be
n_\ell =\frac 1{4\b}(2\b-C+\l)-\frac {C\ell}{2\b}.
\ee
The constraint in \eqv(constraint.1) then yields 
\bea
2^{N-1}
&=&
\sum_{\ell=0}^{K^*-1}\left( \frac 1{4\b}(2\b-C+\l)-\frac {C\ell}{2\b}\right)\nonumber\\
&=&\frac {K^*}{4\b}(2\b-C) -\frac{C}{4\b}K^*(K^*-1) + \l \frac {K^*}{4\b}.
\eea
Hence 
\be
\l = (C-2\b) + C(K^*-1)+\b\frac{2^{N+1}}{K^*}.
\ee
Therefore
\be\Eq(nell)
n_\ell =\frac 1{4\b} \left((K^*-1)C +\b\frac{2^{N+1}}{K^*}-2C\ell\right).
\ee

Finally, we can compute $C$. We have
\bea
\sum_{\ell=0}^{K^*-1}  \ell n_\ell
&=& \sum_{\ell=0}^{K^*-1} \ell
\frac 1{4\b} \left((K^*-1)C +\b \frac{2^{N+1}}{K^*}-2C\ell\right)\nonumber\\
&=&\frac1{8\b} \left((K^*-1)C +\b\frac{2^{N+1}}{K^*}\right) K^*(K^*-1)
- \frac C{2\b} \sum_{\ell=0}^{K^*-1}\ell^2\nonumber\\
&=&\frac1{8\b} \left((K^*-1)C +\b\frac{2^{N+1}}{K^*}\right) K^*(K^*-1)
\nonumber\\
& &-\frac C{12\b} (K^*-1)K^*(2K^*-1)\nonumber\\
&=&- \frac C{24\b}K^*(K^*-1)(K^*+1) + (2^{N-2})( K^*-1).
\eea
Hence 
\be 
C
= \frac {\e^22^{N-2}K^*}{(2^{N-1}-1)^2+\frac{\e^2  K^*(K^*-1)(K^*+1)}{24\b}}.
\ee
For $N$ and $K^*$ large, this is up to lower order terms
\be
C \approx \e^2\frac {K^*}{2^N+\e^2(K^*)^32^{-N}/6\b}
= \e^2\frac {K^*2^{-N}}{1+\chi(K^*)^32^{-2N}},
\ee
with $\chi=\e^2/6\b$.
We next compute the cost of these configurations, with the approximation made so far.  Then 
\be
S_{\spr}(K^*) =C^2 \e^{-2}( 2^{N-1}-1)^2
\approx \frac {3\chi\b}{2}\frac{ (K^*)^2}{\left(1+\chi (K^*)^32^{-2N}\right)^2}.
\ee
With $n_\ell$ from \eqv(nell),
\bea
\sum_{\ell=0}^{K^*-1} n_\ell(n_\ell-1)
&\approx& 
\frac 1{(4\b)^2}\left(\frac 13 C^2(K^*)^3+\frac{2^{2N+2}\b^2}{K^*}\right) -2^{N-1}\nonumber\\
&\approx&
\frac 34\frac{\chi^2 2^{2N} (K^*)^5}{\left(2^{2N}+\chi (K^*)^3\right)^2}
+2^{2N-2}(K^*)^{-1}-2^{N-1}.
\eea
Set $K^*=\a 2^{2N/3}$. Then the total cost $S(\alpha)$ becomes
\be 
S(\a)\approx 2^{4N/3} \b \left(\frac 32
\frac { \chi^2\a^5+ \chi\a^2 }{(1+\chi \a^3)^2}
+\frac 1{2\a} \right).
\ee
If we further substitute $\a=\chi^{-1/3} u$, then this becomes 
\be
S(u)\approx 2^{4N/3} \b \chi^{1/3}f(u),
\ee
with 
\be
f(u)=\frac 32\frac {u^5+u^2}{(1+u^3)^2} +\frac 1{2u}
= \frac 32\frac {u^2}{1+u^3} +\frac 1{2u}.
\ee
For the derivative, we get
\be
f'(u)=-\frac{(1-2u^3)^2}{u^2(u^3+1)^2}.
\ee
On the real numbers, $f$ is decreasing and has a saddle point at $u=2^{-1/3}$. It looks disturbing that this expression tends to zero as $u\to\infty$. However, if $u$ becomes too large, then the constraint that all $n_\ell\geq 0$ will become violated. In fact, 
\be
n_{K^*}\approx \frac 1{4\b} \left(-\frac {\e^2 (K^*)^2}{2^N+\chi (K^*)^32^{-N}} + 2^{N+1}\b (K^*)^{-1}\right)
= \frac 1{4\b}2^{N/3} \left(- \frac {\e^ 2 \a^2}{1+\chi \a^3} + \frac {2\b}{\a}\right).
\ee
This is zero if 
\be
{3\chi \a^3}=1+\chi\a^3,
\ee
or $\a^3= 1/(2\chi)$. This limits the range of $\a$ to $\a^3\leq 1/(2\chi)$. Interestingly, this is precisely also the optimal range. 
Then 
\be 
f(2^{-1/3})=2^{1/3}.
\ee
The total cost for the optimiser is therefore
\be
S\approx 2^{4N/3} \b \chi^{1/3}2^{1/3}.
\ee
This proves the lower bound in \eqv(main.2).
\end{proof}
We see that the optimal configuration for this bound is supported on the interval of width
$\e 2^{2N/3} /(2\chi)^{1/3}=2^{2N/3}(3\e\b)^{1/3}$.

\subsection{Upper bound on the total cost}
\label{ss.optfull}

\begin{lemma}
\TH(upperbound.1)
Fix $N\in \N$ sufficiently large. Then 
\be
\Eq(upperbound.2)
\min_{h} S(h,\T^{(N)})\leq 2^{4N/3} (\b\e)^{2/3} \left(\tfrac{16}{3}\right)^{1/3}
\left(\tfrac43 +\tfrac28 \left(\tfrac{16}{3}\right)^{1/3}\right).
\ee
\end{lemma}

\begin{proof} 
The strategy of the proof of this upper bound is to exhibit a configuration that has precisely this cost. This configuration is obtained as follows. First, fix $1 \leq M\leq N$. At time $M$, place the $2^M$ particles at distances $\e$. We have already computed the associated spreading cost by solving the corresponding Dirichlet problem. From time $M$ to time  $N$, we let newborn particles never move and stay where they are born. Thus, no further spreading cost is incurred, and the interaction cost can be computed exactly. Finally, we optimise over $M$.

We need the following lemma.

\begin{lemma}
\TH(almost-opt.1)
Let 
\be
\Eq(linear.101)
u(\bz^{(M)}) = \e\left(\sum_{\ell=1}^{M} 2^{M-\ell} z_\ell-2^{M-1}+\tfrac12\right).
\ee
Let $h^{\Dir}(\bz)$ be the solution to the Dirichlet problem in \eqv(diri.1) with boundary condition $u$. Set
\be
h^{**}(\bz^{(n)}) = \begin{cases} 
h^{\Dir}(\bz^{(n)}), &0 \leq n\leq M,\\
u(\bz^{(M)}), &n \geq M.
\end{cases}
\ee
Then
\be
S(h^{**},\T^{(N)}) = S_{\spr}(h^{\Dir},\T^{(M)}) + S_{\rep}(h^{\Dir},\T^{(M)})
\ee
with
\be 
\begin{aligned}
S_{\spr}(h^{\Dir},\T^{(M)}) &= \e^2\,2^{2M-3} + O(M^2 2^{M}),\\
S_{\rep}(h^{\Dir},\T^{(M)}) &= \b\,\tfrac23\,\left(2^{2N-M+1}+2^M-3\,2^N\right).
\end{aligned}
\ee
\end{lemma}

\begin{proof}
For the configuration $h^{**}$, all particles stay at a distance at least $\e$ up to time $M$. The spreading cost up to that time is $S_{\spr}(h^{\Dir},\T^{(M)})$. After that time, particles only replicate and do not move. Thus, no further spreading cost occurs. Also, the particles at different sites do not interact, since their distance is $\e$. At time $M+n$, there are $2^n$ particles at the occupied sites, so that the total repulsion cost is
\bea
&&S_{\rep}(h^{\Dir},\T^{(M)}) = \b\, 2^M\sum_{n=1}^{N-M} 2^n(2^n-1) \nonumber\\
&&\quad = \b\, 2^M\tfrac23\left(2^{2(N-M)+1} + 1 - 3\,2^{N-M}\right)
= \b \tfrac 23\left( 2^{2N-M+1}- 3\,2^{N}+ 2^M\right).
\eea
\end{proof}

Finally, we optimise over $M$. It is convenient to set $M=\tfrac{2N}{3}+\a$. Ignoring integer constraints and keeping only the leading order terms, we get that the optimal $\a$ is
\be
\a =\frac {\ln \left(\frac {16\b\e}{3}\right)}{3\ln 2}.
\ee
Thus, the width of the configuration is $2^{2N/3+1}\left(\frac{2\b}{3\e^2}\right)^{1/3}$. This gives the upper bound in \eqv(upperbound.2).
\end{proof}

\begin{proof}[Proof of Theorem \thv(main.1)]
Theorem~\thv(main.1) follows directly from Lemmas~\ref{lowerbound.1}--\ref{upperbound.1}.
\end{proof} 
We see that the shapes of the optimal configurations in the upper and the lower bounds are different. The flat profile in the upper bound is clearly not the correct shape, whereas the tent-like profile in the lower bound may be. 


\section{Bounds on the spreading. Proof of Theorem \thv(main.3)}
\label{ss.bdspreading}

\begin{proof}
We prove the three parts of the theorem in the order (i), (iii) and (ii).

\medskip\noindent
(i) The lower bound on the spreading is straightforward from the fact that the repulsion cost is minimal for the flat profile, and the fact that when the spreading is smaller the repulsion cost of the flat profile, see Lemma \thv(flat.1), is too large compared to the upper bound from Lemma \thv(upperbound.1).

\medskip\noindent
(iii) It is enough to consider the cost to place a single particle at a position $\geq R$. Since the distribution of a single particle at time $N$ is Gaussian with variance $N$, the latter event has a probability smaller than $\exp(-R^2/2N)$. Comparing with the upper bound on the total cost from Lemma \thv(upperbound.1), we find that 
\be 
R \leq 2^{2N/3}(\b\e)^{1/3}\sqrt{2NC},
\ee 
with $C$ the constant in \eqv(upperbound.2).

\medskip\noindent
(ii) We revisit the lower bound obtained from the lower bounds on the spreading costs and the repulsion cost, but impose the condition that there is a fraction $\d$ of the $2^N$ particles at time $N$ beyond the 
a fixed range $K$. We consider the situation where $2^N (1-\d)$ are located on the sites up to $K^*$ while the remaining particles are between $K$ and $K+L$. Both $K^*$ and $L$ will be optimised at the end.
Thus, we now have two constraints, 
\be\Eq(constraint.101)
\sum_{k=0}^{K^*-1} n_k=(1-\d)\,2^{N-1}, \qquad \sum_{k=K}^{K+L} n_k=\d\,2^{N-1},
\ee
with $n_k\geq 0$ for all $0 \leq k \leq K^*-1$ and $K \leq k \leq K+L$. We get a lower bound on the total energy by minimising 
\be
\e^2  \frac{\left(\sum_{k=0}^{{K+L}} k n_k\right)^2}{(2^{N-1}-1)^2}+ 2\b \sum_{k=0}^{{K+L}} n_k(n_k-1)
\ee
over $n_k$ and $K^*$ under the two constraints above. 

We first fix $K^*$ and minimise with respect to $n_k$. This leads to the set of equations 
\be
\e^2 \frac{\left(\sum_{k=0}^{{K+L}}  k n_k\right)(2\ell +1)}
{(2^{N-1}-1)^2}+ 2\b(2n_\ell-1) -\l_1=0
\ee
for $\ell=0,\dots,K^*-1$ and 
\be
\e^2 \frac{ \left(\sum_{k=0}^{{K+L}}  k n_k\right)(2\ell +1)}
{(2^{N-1}-1)^2}+ 2\b(2n_\ell-1) -\l_2=0
\ee
for $\ell\geq K^*$, where $\lambda_1$, $\lambda_2$ are Lagrange multipliers. We write these equations in the form 
\be
C (2\ell +1)+4\b n_\ell-2\b-\l_i=0, \quad i=1,2,
\ee
with 
\be
C = \e^2 \frac{\sum_{k=0}^{{K+L}} kn_k}{(2^{N-1}-1)^2}.
\ee
Thus,
\be
n_\ell =\frac 1{4\b}(2\b-C+\l_i)-\frac {C\ell}{2\b}, \quad i=1,2. 
\ee
For $\l_1$ the first constraint in \eqv(constraint.101) yields 
\bea
2^{N-1}(1-\d)
&=&
\sum_{\ell=0}^{K^*-1}\left( \frac 1{4\b}(2\b-C+\l_1- {2C\ell}\right)\nonumber\\
&=&\frac {K^*}{4\b}(2\b-C) - \frac{C}{4\b} K^*(K^*-1) + \l_1 \frac{K^*}{4\b},
\eea
or 
\be
\l_1 = (C-2\b) + C(K^*-1)+\b\frac{(1-\d)2^{N+1}}{K^*}.
\ee
Therefore
\be\Eq(nell*)
n_\ell =\frac 1{4\b} \left(C(K^*-1)+\b\frac{(1-\d)2^{N+1}}{K^*}-2C\ell\right).
\ee
For $\l_2$ the second constraint in \eqv(constraint.101) yields   
\be
\d\, 2^{N-1}=\sum_{\ell=K}^{K+L} \left(2\b-C+\l_2-2C\ell\right),
\ee
ot
\be
\l_2= \frac {\b2^{N+1}\d}{L } -\frac {2\b-C}{L}+\frac{C}{L} L(L+2K).
\ee
So, for $\ell\geq K$,
\be
n_\ell=\frac  {1}{4\b}\left(\frac{\b 2^{N+1}\d}L+C(L+2K)-2C\ell\right).
\ee

Finally, we compute $C$, keeping only the leading order terms. We have 
\be
\sum_{\ell=0}^{K^*-1} \ell n_\ell = 2^{N-2}(1-\d)K^*-\frac C{24 \b} (K^*)^3
\ee
and 
\bea
&&\sum_{\ell=K}^{K+L} \ell n_\ell
\nonumber\\
&&= \frac {1}{4\b}\left((L+2K)\,\b\d 2^{N} +C\left(\tfrac12(L^2+2KL)(L+2K)-\tfrac23(L^3+3L^2K+3LK^2)\right)\right)\nonumber\\
&&= \frac{1}{4\b}\left((L+2K)\,\b\d 2^{N} -C\tfrac16 L^3\right).
\eea
This gives 
\bea
C&=&\frac{\frac{\e^2}{2^{2N-2}}\left((1-\d)2^{N-2}K^*+\d 2^{N-1}(L+2K)\right)}{1+
\frac{\e^2}{2^{2N-2}}\frac 1{24\b}\left(L^3+(K^*)^3\right)}\nonumber\\
&=&\e^2\frac{(1-\d)2^{N-2}K^*+\d 2^{N-1}(L+2K)}{2^{2N-2}+
\frac {\e^2}{24\b}\left(L^3+(K^*)^3\right)}\nonumber\\
&=&\e^2\frac{(1-\d)2^{-N}K^*+\d 2^{-N-1}(L+2K)}{1+
\frac {\e^2}{6\b}\left(L^3+(K^*)^3\right)2^{-2N}}.\nonumber
\eea
Note that the positivity constraint implies that 
\be
L \leq \sqrt{\frac {\b\d 2^{N+1}}C}
\ee
and
\be
K^* \leq \sqrt{\frac{2\b(1-\d)2^N}{C}}.
\ee

For the repulsion cost we get
\bea
S_{\rep} &=&\frac{2\b}{(4\b)^2}\left(\frac {2^{2N}4\b^2(1-\d)^2}{K^*}+\frac {4\b\d^2 2^{2N}}{L}
+\frac13 C^2\left(L^3+(K^*)^3\right)\right) \\ \nonumber
&=& \b\, 2^{2N-1}\left(\frac {(1-\d)^2}{K^*}+\frac{\d^2}{L}\right)
+\frac32\, \b\,\chi^2 2^{-2N}(L^3+(K^*)^3)
\frac{\left((1-\d)K^*+\d 2^{-1}(L+2K)\right)^2}
{\left(1+\chi(L^3+(K^*)^3)\,2^{-2N}\right)^2}.
\eea
For the spreading cost we get 
\be
S_{\spr} = \frac{3\chi\b}{2}\frac{\left((1-\d)K^*+\d 2^{-1}(L+2K)\right)^2}{\left(1+\chi(L^3+(K^*)^3)2^{-2N}\right)^2}.
\ee
Set $K^*=\a 2^{2N/3}$, $L=\l2^{2N/3}$, $K=\g 2^{2N/3}$. Then the total interaction takes the form
\be
S = 2^{4N/3}\left(\frac{\b}{2}\left(\frac {(1-\d)^2}\a+ \frac {\d^2} {\l}\right)
+\frac {  3 \chi\b \left((1-\d)\a+ \d(\l+2\g)\right)^2}
{2\left(1+\chi(\a^3+\l^3)\right)}\right).
\ee
Setting $u=\chi^{-1/3}\a,v=\chi^{-1/3}\l,w=\chi^{-1/3}\g$, we see that this can be written as 
\be
S = 2^{4N/3} \b\chi^{1/3} f(u,v,w)
\ee
with
\be
f(u,v,w) = \left(\frac{(1-\d)^2}{2u} + \frac {\d^2}{2v} + \frac32\frac{((1-\d)u+\d(v+2w))^2}{1+u^3+v^3}\right).
\ee
Differentiating $f$ with respect to $u$, we get 
\be
\frac{\del f}{\del u}
= -\frac{\left(-(1-\d)v^3+3u^2\d(v+2w)+(\d-1)+2u^3\right)^2}{2u^2(v^3+u^3+1)^2}.
\ee
Thus, $f$ is monotone in $u$, and $u$ again takes the maximal allowed value. Similarly, $\frac{\del f}{\del v}$ is negative, so $\l$ attains its maximum allowed value. In particular, it follows that $\l=\sqrt{\d/(1-\d)}\,a$. 

Using the above, we obtain an equation for $u$, namely,
\be
u^3\left(2-3\d + \frac{3\d^{3/2}}{\sqrt{1-\d}}-\frac{\d^{3/2}}{(1-\d)^{3/2}}\right) + 6w\, \d u^2 = 1.
\ee
For small $\d$, this reads 
\be
2u^3 + 6w\,\d u^2 \sim 1.
\ee 
The latter equation has a positive solution that, for large values of $\d w$, behaves as
\be
u \sim \frac{1}{\sqrt{6\d\, w}}\left(1-\frac12 (6\d\,w)^{-3/2}\right).
\ee
We deduce that there is a constant $c>0$ such that 
\be
S \geq 2^{4N/3}\,\b\chi\, c\d w,
\ee
which, for $\d w>C$, implies that such a configuration cannot be a minimiser. Thus, for any minimising configuration, $\d w\leq (C/\beta\chi c) (\beta\e)^{2/3}$, i.e., $\e K \leq 2^{2N/3} (\beta\e)^{1/3} 6^{2/3} (C/c\delta)$, where we recall that $\chi = (\e^2/6\beta)$. Since the particles are spread out over a width $\e K$, the claim in (ii) follows.
\end{proof}

\end{document}